\documentclass[10pt,reqno]{amsart}

\usepackage{amsmath}
\usepackage{amssymb}
\usepackage{amsthm}
\usepackage{amsfonts}
\usepackage{arydshln}
\usepackage[thinlines]{easybmat}
\numberwithin{equation}{section}
\newcommand{\h}{\mathcal{H}}
\newcommand{\ka}{\mathcal{K}}
\newcommand{\Kt}{K_\Theta}
\newcommand{\bh}{B({\mathcal H})}

\def\kda{K_\alpha}
\def\kdt{K_\theta}
\newtheorem{claim}{claim}[section]
\newtheorem{theorem}[claim]{Theorem}

\newtheorem{lemma}[claim]{Lemma}
\newtheorem{proposition}[claim]{Proposition}

\newtheorem{corollary}[claim]{Corollary}

\theoremstyle{definition}

\newtheorem{example}[claim]{Example}

\title{Conjugations in $L^2$ and their invariants}
\author[M. C. C\^amara, K. Kli\'s--Garlicka, B. \L anucha, and M. Ptak]{M. Cristina C\^amara, Kamila Kli\'s--Garlicka, Bartosz \L anucha, and Marek Ptak}
\thanks{The work of the first author was partially supported by FCT/Portugal
through UID/MAT/04459/2013 and the research of the second and the fourth authors was financed by the Ministry of Science and Higher Education of the Republic of Poland}

\address{M. Cristina C\^amara, Center for Mathematical Analysis, Geometry and Dynamical Systems, Mathematics Department, Instituto Superior T'ecnico, Universidade de Lisboa, Av. Rovisco Pais, 1049--001 Lisboa, Portugal}
\email{ccamara@math.ist.utl.pt}
\address{Kamila Kli\'s-Garlicka, Department of Applied Mathematics, University of Agriculture, ul. Balicka 253c, 30-198 Krak\'ow, Poland}
\email{rmklis@cyfronet.pl}
\address{Bartosz \L anucha, Department of Mathematics,  Maria Curie-Sk\l odowska University, Maria Curie-Sk\l o\-dow\-ska Square 1, 20-031 Lublin, Poland}
\email{bartosz.lanucha@poczta.umcs.lublin.pl}
\address{Marek Ptak, Department of Applied Mathematics, University of Agriculture, ul. Balicka 253c, 30-198 Krak\'ow, Poland}
\email{rmptak@cyf-kr.edu.pl}

\begin{document}
\begin{abstract}{Conjugations in  space $L^2$ of the unit circle commuting with multiplication by $z$ or intertwining multiplications by $z$ and $\bar z$  are characterized. We also study their behaviour with respect to the Hardy space, subspaces invariant for the unilateral shift and model spaces.}
\end{abstract}
\maketitle

\section{Introduction}
Let $\h$ be a complex Hilbert space and denote by $\bh$ the algebra of all bounded linear operators on $\h$. A {\it conjugation} $C$  in $\h$ is an antilinear isometric involution, i.e., $C^2=id_{\h}$ and
\begin{equation}\label{e1}
  \langle Cg,Ch \rangle=\langle h,g\rangle \quad \text{ for } g,h\in\h.
\end{equation}
Conjugations have recently been intensively studied and the roots of this subject comes from physics.
An operator $A\in\bh$ is called {\it $C$--symmetric} if $CAC=A^*$ (or equivalently $AC=CA^*$).
A strong motivation to study conjugations comes from the study of complex symmetric operators, i.e., those operators that are $C$--symmetric with respect to some conjugation $C$. For references see for instance \cite{GP,GPP, GP2, CKP, CFT, KoLee18}. Hence obtaining the full description of  conjugations with certain properties is of great interest.

Let $\mathbb{T}$ denote the unit circle, and let $m$ be the normalized Lebesgue measure on $\mathbb{T}$. Consider the spaces $L^2=L^2(\mathbb{T,}m)$, $L^{\infty}=L^\infty(\mathbb{T},m)$, the classical Hardy space $H^2$ on the unit disc $\mathbb{D}$ identified with a subspace of $L^2$, and the Hardy space $H^\infty$ of all analytic and bounded functions in $\mathbb{D}$ identified with a subspace of $L^\infty$.
Denote by $M_\varphi$ the operator defined on $L^2$ of  multiplication by a function $\varphi\in L^\infty$.

The most natural conjugation in $L^2$ is $J$ defined by $Jf=\bar f$, for $f\in L^2$. This conjugation has two natural properties:  the operator $M_z$ is $J$--symmetric, i.e., $M_z J=J M_{\bar z}$, and
$J$ maps an analytic function into a co-analytic one, i.e., $J H^2=\overline{H^2}$.

Another natural conjugation in $L^2$ is $J^{\star}f=f^{\#}$ with $f^{\#}(z)\overset{df}{=}\overline{ f(\bar z)}.$ The conjugation $J^{\star}$ has a different behaviour: it commutes with multiplication  by $z$ ($M_z J^\star=J^\star M_{ z}$) and leaves analytic functions invariant, $J^\star H^2\subset{H^2}$.
The map $J^{\star}$ appears for example in connection with Hankel operators (see \cite[pp. 146--147]{MAR}). Its connection with model spaces was studied in \cite{CGW} (Lemma 4.4, see also \cite[p. 37]{berc}). Hence a natural question is to characterize conjugations with respect to these properties. The first step was done in \cite{CKP} where all conjugations in $L^2$ with respect to which  the operator $M_z$ is $C$--symmetric  were characterized, see Theorem \ref{c11}.
In Section 2 we give a characterization of all
 conjugations which commute with $M_z$, Theorem \ref{c1}. In Section 3, using the above characterizations we show that there are no conjugations in $L^2$ leaving $H^2$ invariant, with respect to which the operator $M_z$ is $C$--symmetric. We also show  that $J^{\star}$ is the only conjugation commuting with $M_z$ and leaving $H^2$ invariant.

   Beurling's theorem makes  subspaces of $H^2 $ of the type $\theta H^2$ ($\theta$ inner function,  i.e., $\theta\in H^\infty$, $|\theta|=1$ a.e. on $\mathbb{T}$) exceptionally interesting, as the only invariant subspaces for the unilateral shift $S$, $Sf(z)=zf(z)$ for $f\in H^2$. On the other hand, model spaces (subspaces of the type $K_\theta=H^2\ominus \theta H^2$), which are invariant for the adjoint of the unilateral shift, are important in model theory, \cite{NF}.

In \cite{CKP} all conjugations $C$  {with respect to which} the operator $M_z$ is $C$--symmetric  and mapping a model space $K_\alpha$ into another  {model space} $K_\theta$ were  characterized, with the assumption that $\alpha$ divides $\theta$ ($\alpha\leqslant\theta$). {Recall that for $\alpha$ and $\theta$ inner, $\alpha\leqslant\theta$ means that $\theta/\alpha$ is also an inner function.} In what follows we will show that the result holds without the assumption $\alpha\leqslant\theta$.

In Section 4 conjugations {commuting with $M_z$ and} preserving model spaces are described. Section 5 is devoted to conjugations between $S$--invariant subspaces (i.e., subspaces of the form $\theta H^2$ with $\theta$ {an} inner function). In the last section we deal with conjugations commuting with the truncated shift $A^\theta_z$ {($A^\theta_z=P_{\theta}M_{z|K_{\theta}}$ where $P_{\theta}$ is the orthogonal projection from $L^2$ onto $K_{\theta}$)} or
conjugations  such that $A^\theta_z$ is $C$--symmetric with respect to them.

\section{$M_z$ and $M_z$--commuting conjugations in $L^2$.}

 Denote by $J$ the conjugation in $L^2$ defined as $Jf=\bar f$, for $f\in L^2$. This conjugation has the following obvious properties:
 \begin{proposition}\label{J}
\begin{enumerate}
\item $M_z J=J M_{\bar z}$;
  \item $M_\varphi J=JM_{\bar\varphi}$ for all $\varphi\in L^\infty$;
  \item $J H^2=\overline{H^2}$.
\end{enumerate}
\end{proposition}

  Let us consider all  conjugations $C$ in $L^2$ satisfying the condition
\begin{equation}\label{mz}
  M_z C=C M_{\bar z}.
\end{equation}
Such conjugations were  studied in \cite{CKP} and are called {\it $M_z$--conjugations}. The following theorem characterizes all $M_z$--conjuga\-tions in $L^2$.
\begin{theorem}[\cite{CKP}]\label{c11}
Let $C$ be a conjugation in $L^2$. Then the following are equivalent:
\begin{enumerate}
\item $M_z C=C M_{\bar z}$, 
\item $M_\varphi C=C M_{\bar\varphi}$ for all $\varphi\in L^\infty$,
 \item there is $\psi\in L^\infty$, with $|\psi|=1$, such that $C=M_\psi J$,
 \item there is $\psi^\prime\in L^\infty$, with $|\psi^\prime|=1$, such that $C=JM_{\psi^\prime}$.
 \end{enumerate}
\end{theorem}
Another natural conjugation in $L^2$ is defined as
\begin{equation}\label{djstar}J^{\star}f=f^{\#}\ \text{ with }\ f^{\#}(z)\overset{df}{=}\overline{ f(\bar z)}.\end{equation}
The basic properties of $J^{\star}$ are the following:
\begin{proposition}\label{Jstar}
\begin{enumerate}
\item $M_z J^{\star}=J^{\star} M_{ z}$;
\item $M_{\bar z} J^{\star}=J^{\star} M_{\bar z}$;
  \item $M_\varphi J^{\star}=J^{\star}M_{\varphi^{\#}}$ for all $\varphi\in L^\infty$;
  \item $J^\star H^2=H^2$.
\end{enumerate}
\end{proposition}

In the context of the
Proposition \ref{Jstar} it seems natural to consider all conjugations $C$ in $L^2$
commuting with $M_z$, i.e.
\begin{equation}\label{eq1}M_z C=C M_{ z}.
\end{equation}
Such conjugations  will be called {\it $M_z$--commuting}. In what follows we will often deal with  functions
 $f\in L^2$  such that  $f(z)=f(\bar z)$ {a.e. on $\mathbb{T}$}, which will be called {\it symmetric}. Observe that if $f$ is symmetric and $f\in H^2$, then we also have $f\in\overline{H^2}$ and so it is a constant function.

\begin{theorem}\label{c1}
Let $C$ be a conjugation in $L^2$. Then the following are equivalent:
\begin{enumerate}\item $ M_z C=C M_{z}$ ,
\item $M_\varphi C=C M_{\varphi^{\#}}$ for all $\varphi\in L^\infty$,
 \item there is a symmetric unimodular function $\psi\in L^\infty$ 
 such that $C=M_\psi J^{\star}$,
     \item there is a symmetric unimodular function $\psi^\prime\in L^\infty$ 
      such that $C=J^{\star}\,M_{\psi^\prime}$.
 \end{enumerate}
\end{theorem}
\begin{proof} We will show that $(1)\!\Rightarrow \!(3)$. \!The other implications are straightforward.
Assume that $M_{ z}C =CM_z$. Then $M_z CJ^{\star}=CM_{
z}J^{\star}=CJ^{\star}M_z$. It follows  that the linear operator
$CJ^{\star}$ commutes with $M_z$. By \cite[Theorem 3.2]{RR}, there is $\psi\in L^\infty$
such that $CJ^{\star}=M_{\psi}$. Hence
$C=M_{\psi}J^{\star}=J^{\star}M_{\psi^{\#}}$.

By \eqref{e1} for any $f,g\in L^2$ we have
\begin{align*}\int g\bar{f}\,dm&=\langle g, f\rangle=
\langle Cf,Cg\rangle\\
&=\langle \psi f^{\#},\psi g^{\#}\rangle=\int|\psi(z)|^2\,\overline{f(\bar z)}g(\bar z)\,dm(z)\\
                               &=\int|\psi(\bar z)|^2\,\overline{f( z)}g( z)\,dm(z).
\end{align*}
Hence $|\psi|=1$ a.e. on $\mathbb{T}$. On the other hand, since $C^2=I_{L^2}$,  for all $f\in L^2$ we have
$$f=C^2 f=M_\psi J^{\star} M_\psi J^{\star} f=M_\psi J^{\star}({\psi f^{\#}})=\psi\psi^{\#} f,$$
which implies that $\psi\psi^{\#}=1$ a.e. on $\mathbb{T}$.  Therefore $\psi$ is symmetric, i.e., $\psi(z)=\psi(\bar z)$ a.e. on $\mathbb{T}$.

%
%
\end{proof}
\section{Conjugations preserving $H^2$.}
In the previous section all $M_z$--conjugations $C$ in $L^2$, i.e.,  such that
\begin{equation*}\label{orr}
M_{ z} C=C M_{\bar z}
\end{equation*} or $M_z$-- commuting conjugations, i.e. such that
\begin{equation*}\label{orr1}
M_{z} C=C M_{ z}
\end{equation*}
were characterized.
Let us now consider the question which of them preserve $H^2$. Clearly, if $C$ is a conjugation in $L^2$ and $C(H^2)\subset H^2$, then $C(H^2)= H^2$. Since $J^\star$ preserves $H^2$, it can be considered as a conjugation in $H^2$.
The following result shows that $J^\star$ is in that sense unique.
\begin{corollary}\label{co1}
Let $C$ be an $M_z$--commuting conjugation in $L^2$. If $C(H^2)\subset H^2$, then $C=\lambda J^{\star}$ for some $\lambda\in\mathbb{T}$.
\end{corollary}
\begin{proof}
	By Theorem \ref{c1} we have that $C=M_{\psi}J^{\star}$ for some $\psi\in L^{\infty}$ with $|\psi|=1$ and $\psi(z)=\psi(\overline{z})$ a.e. on $\mathbb{T}$. Since $C$ preserves $H^2$, we have
	$$\psi=M_{\psi}J^{\star}(1)=C(1)\in L^{\infty}\cap H^2=H^{\infty}.$$
 Thus $\psi$ is analytic. Since	it is symmetric, it is also co-analytic. Hence $\psi$ must be a constant function, so  $\psi=\lambda\in\mathbb{C}$ and $|\lambda|=|\psi|=1$.
\end{proof}
\begin{corollary}
There are no $M_z$--conjugations in $L^2$ which preserve $H^2$.
\end{corollary}
\begin{proof}
	If $C$ is an $M_z$--conjugation in $L^2$, then by Theorem \ref{c11} it follows that $C=M_{\psi}J$ for some $\psi\in L^{\infty}$ with $|\psi|=1$. As in the proof of Corollary \ref{co1} the assumption $C(H^2)\subset H^2$ implies that $\psi\in H^{\infty}$, which in turn means that $\psi$ is an inner function. Moreover, for  $n=0,1,2,\dots$ we have
	$$0=\langle C z^{n+1},\overline{z}\rangle=
	\langle \psi\overline{z}^{n+1},\overline{z}\rangle=\langle \psi,z^n\rangle=0.$$
So $\psi=0$ which is a contradiction.
\end{proof}

The following example shows that not all conjugations in $L^2$ satisfy 
either \eqref{mz} or \eqref{eq1}.

\begin{example}There is a set of naturally defined conjugations. For $k,l\in\mathbb{Z}$, $k<l$, define $C_{k,l}\ \colon\ L^2\rightarrow L^2$ by
\begin{equation}\label{ckl}
C_{k,l}\Big(\sum_{n\in\mathbb{Z}}a_n z^n\Big)=\overline{a}_l z^k+\overline{a}_k z^l+\sum_{n\notin\{k,l\}}\overline{a}_n z^n,
\end{equation}
where $\{z^n\}$ is the standard basis in $L^2$. Then \eqref{mz} and  \eqref{eq1} are not satisfied since
$$M_{ z} C_{k,l}(z^k)=M_z(z^l)=z^{l+1},\quad C_{k,l} M_{\bar
z}(z^k)=C_{k,l} (z^{k-1})=z^{k-1}$$ and
$$C_{k,l} M_{z}(z^k)=C_{k,l} (z^{k+1})=\begin{cases}z^{k+1}&\text{if}\ k+1\neq l,\\z^k&\text{if}\ k+1=l.\end{cases}$$
Note that, on the other hand, $C_{k,l}$ preserves $H^2$ whenever $k\geqslant 0$ or
$l<0$.
\end{example}

\section{Conjugations preserving model spaces}
There is another class of conjugations in $L^2$ which appear naturally  in connection with model spaces.
For a nonconstant inner function $\theta$, denote by $K_\theta$ the so called {\it model space} of the form $H^2\ominus\theta H^2$.
 The conjugation $C_\theta$ defined in $L^2$ by
$$C_\theta f=\theta \bar z\bar f$$
has the important property that it preserves the model space $K_\theta$, i.e., $C_\theta K_\theta=K_\theta$. Thus $C_\theta$ can be considered as a conjugation in $K_\theta$.
Such conjugations are important in connection with truncated Toeplitz operators (see for instance \cite{GMR}).
Here we present several  simple properties of such conjugations, which we will use later.
\begin{proposition}\label{ppp} Let $\alpha,\beta,\gamma$ be nonconstant inner functions. Then
\begin{enumerate}
  \item $C_\beta C_\alpha=M_{\beta\bar \alpha}$,
  \item $M_{\gamma} C_\alpha M_{\bar\gamma}$ is a conjugation in $L^2$,
  \item $C_\beta\, M_\gamma = M_{\bar\gamma}\,C_\beta$.
\end{enumerate}
\end{proposition}


Now we will discuss relations between $M_z$--conjugations and model spaces.
The theorem below was proved in \cite[Theorem 4.2]{CKP} with the additional assumption that $\alpha\leqslant \theta$. As we prove here, this assumption is not necessary.

\begin{theorem}\label{t1}
Let $\alpha, \gamma,  \theta$ be inner functions ($\alpha,\theta$ nonconstant).
Let $C$ be a conjugation in $L^2$ such that $M_z C=C M_{\bar z}$. Assume that  $C(\gamma\kda)\subset\kdt$.  Then there is an inner function $\beta$ such that $C=C_\beta$, with $\gamma\alpha\leqslant\beta\leqslant\gamma \theta$ and $\alpha\leqslant \theta$.
\end{theorem}

\begin{proof} Recall the standard notation for the reproducing kernel function at $0$ in $\kda$, namely, $k_0^\alpha=1-\overline{\alpha(0)}{\alpha}$ and its conjugate $\tilde k_0^\alpha=C_\alpha k_0^\alpha=\bar z(\alpha-\alpha(0))$.
By Theorem \ref{c11} we know that $C=M_\psi J$ for some function $\psi\in L^\infty$, $|\psi|=1$. Hence
$$\kdt \ni C (\gamma\tilde k_0^\alpha)= M_\psi J (\gamma\tilde k_0^\alpha)=\psi \overline{\gamma\bar z(\alpha-\alpha(0))}=\bar\gamma\bar\alpha z\psi(1-\overline{\alpha(0)}\alpha).$$
Thus there is $h\in \kdt$ such that $h=\bar\gamma\bar\alpha z\psi(1-\overline{\alpha(0)}\alpha)$.
Since $(1-\overline{\alpha(0)}\alpha)^{-1}$ is a bounded analytic function, we have
$$\bar \gamma\bar\alpha z\psi=h(1-\overline{\alpha(0)}\alpha)^{-1}\in H^2.$$
Since $\bar\gamma\bar\alpha z\psi\in H^2$ and $|\bar \gamma\bar\alpha z\psi|=1$ a.e. on $\mathbb{T}$, it has to be an inner function. Moreover $\beta=z\psi$ has to be inner and  divisible by $\gamma\alpha$, i.e., $\gamma\alpha\leqslant\beta$.

On the other hand, we have similarly
 $$\kdt\ni C_\theta C (\gamma k_0^\alpha) =C_\theta(\psi \overline{\gamma (1-\overline{\alpha(0)}\alpha)}=\theta  \gamma\bar z\bar\psi (1-\overline{\alpha(0)}\alpha),$$
 and $\theta\gamma\bar\beta={\theta}{\gamma} \bar z\bar\psi\in H^2$. Hence  $\beta$ divides $ {\theta}{\gamma}$, i.e.,
  $\beta\leqslant \gamma\theta$. It is clear that $C=C_\beta$. Finally, we have $\alpha\leqslant\theta$ as a consequence of $\gamma\alpha\leqslant\beta\leqslant \gamma\theta$.
  \end{proof}
Note that if $\alpha\leqslant \theta$ and $C=C_\beta$ for some inner $\beta$ with $\gamma\alpha\leqslant\beta\leqslant\gamma \theta$, then $K_{\alpha}\subset K_{\tfrac{\beta}{\gamma}}\subset K_{\theta}$, $C_{\beta}M_{\gamma}=C_{\tfrac{\beta}{\gamma}}$ and
$$C(\gamma K_{\alpha})=C_{\beta}M_{\gamma}(K_{\alpha})=C_{\tfrac{\beta}{\gamma}}(K_{\alpha}) \subset K_{\tfrac{\beta}{\gamma}}\subset K_{\theta}.$$ Hence the implication in Theorem \ref{t1} is actually an equivalence.

  The corollary bellow strengthens \cite[Proposition 4.5]{CKP}.
\begin{corollary}\label{bbb} Let $\alpha,  \theta$ be nonconstant inner functions, and
let $C$ be a conjugation in $L^2$ such that $M_z C=C M_{\bar z}$. Assume that  $C(\kda)\subset\kdt$.  Then $\alpha\leqslant \theta$ and there is an inner function $\beta$ such that $C=C_\beta$, with $\alpha\leqslant\beta\leqslant \theta$.
\end{corollary}

Let us turn to discussing the relations between $M_z$--commuting conjugations and model spaces.
The following proposition describes some more properties of $J^\star$.


\begin{proposition} Let $\alpha$ be an inner function. Then
\begin{enumerate}
  \item $J^\star (\alpha H^2)=\alpha^{\#}H^2$;
  \item $J^\star(K_\alpha)=K_{\alpha^{\#}}$;
   \item $J^{\star}C_\alpha =C_{\alpha^{\#}}J^{\star}$.
  \end{enumerate}
\end{proposition}
\begin{proof}
The condition (1) is clear, (2) and (3) were proved in \cite[Lemma 4.4]{CGW}.
\end{proof}

Hence the conjugation $J^{\star}$ has a nice behaviour in connection with model
spaces, namely $J^{\star}(\kda)=K_{\alpha^{\#}}$. Theorem \ref{bb1} below
says that the conjugation $J^{\star}$ is, in some sense the only $M_z$--commuting conjugation with this property.

We start with the following:
\begin{proposition}\label{bb} Let $\alpha, \gamma,  \theta$ be inner functions ($\alpha,\theta$ nonconstant).
Let $C$ be an $M_z$--commuting conjugation in $L^2$. Assume
that  $C(\gamma\kda)\subset\kdt$.  Then $\alpha\leqslant
\theta^{\#}$ and there is an inner function $\beta$ with
$\gamma\alpha\leqslant\beta\leqslant\gamma \theta^{\#}$ such that
$C=J^{\star} M_{\frac{\beta}{\gamma \alpha} \bar\gamma}$.
\end{proposition}

\begin{proof} Observe that since $C$ is an $M_z$--commuting conjugation, taking antilinear adjoints and applying \cite[Proposition 2.1]{CKP}
we get $M_{\bar z} C=C M_{\bar z}$. Since by Proposition \ref{ppp} the antilinear operator ${M_{\gamma}} C_\alpha {M_{\overline{\gamma}}}$ is a conjugation, then
$J^{\star}C{M_{\gamma}} C_\alpha {M_{\overline{\gamma}}}$ is also a conjugation.
Note also that ${M_{{\gamma}}} C_\alpha
{M_{\overline{\gamma}}}M_z = M_{\bar z}{M_{{\gamma}}} C_\alpha {M_{\overline{\gamma}}}$.  Hence
\begin{equation}
  J^{\star}C {M_{{\gamma}}} C_\alpha {M_{\overline{\gamma}}} M_z=  M_{\bar z}J^{\star}C{M_{{\gamma}}}C_\alpha{M_{\overline{\gamma}}}.
\end{equation}
On the other hand,
\begin{displaymath}\begin{split}
J^{\star}C{M_{{\gamma}}} C_\alpha{M_{\overline{\gamma}}}(\gamma\kda)&\subset
J^{\star}C{M_{{\gamma}}} C_\alpha(\kda)\\
& \subset
J^{\star}C(\gamma\kda)\subset J^{\star}(K_{\theta})\subset
K_{\theta^{\#}}.
\end{split}
\end{displaymath}
By Theorem \ref{t1} there is an inner function $\beta$
such that $ J^{\star}C {M_{{\gamma}}}C_\alpha {M_{\overline{\gamma}}}=C_\beta$, with
$\gamma\alpha\leqslant\beta\leqslant\gamma \theta^{\#}$ and
$\alpha\leqslant \theta^{\#}$. Hence $C=J^{\star} C_\beta{M_{{\gamma}}}
C_\alpha {M_{\overline{\gamma}}}$.  Therefore
\[C=J^{\star} M_{\frac{\beta}{\gamma \alpha} \bar\gamma}
\]

\end{proof}

{As in Theorem \ref{t1} the implication in Proposition \ref{bb} can be reversed. Indeed, if $\alpha\leqslant \theta^{\#}$ and $C=J^{\star} M_{\frac{\beta}{\gamma \alpha} \overline{\gamma}}$ for some inner function $\beta$ with $\gamma\alpha\leqslant\beta\leqslant\gamma \theta^{\#}$, then $K_{\alpha}\subset K_{\tfrac{\beta}{\gamma}}\subset K_{\theta^{\#}}$ and
$$C(\gamma K_{\alpha})=J^{\star} M_{\frac{\beta}{\gamma} \overline{\alpha}}(K_{\alpha})=J^{\star} C_{\frac{\beta}{\gamma}}C_{\alpha}(K_{\alpha})\subset J^{\star}(K_{\tfrac{\beta}{\gamma}})\subset J^{\star}( K_{\theta^{\#}})= K_{\theta}.$$}
\begin{theorem}\label{bb1} Let $\alpha, \theta$ be nonconstant inner functions, and let $C$ be an $M_z$--commuting conjugation in $L^2$, i.e., $M_z C=C M_{ z}$. Assume
that  $C(\kda)\subset\kdt$.  Then $\alpha\leqslant \theta^{\#}$ and
$C=\lambda J^{\star} $ with  $\lambda\in \mathbb{T}$.
\end{theorem}
\begin{corollary}\label{ct}
Let $C$ be an $M_z$--commuting conjugation in $L^2$\!. \!Assume that there is some nonconstant inner function $\theta$ such
that $C(\kdt)\subset K_{\theta^{\#}}$. Then $C=\lambda J^{\star}$ with  $\lambda\in \mathbb{T}$.
\end{corollary}
\begin{proof}[Proof of Theorem \ref{bb1} ] By Proposition \ref{bb} there is an inner function $\beta$ with $\alpha\leqslant\beta\leqslant \theta^{\#}$ such that $C=J^{\star} M_{\frac{\beta}{ \alpha} }$. The function $\frac{\beta}{ \alpha}$ is inner and by Theorem \ref{c1} it is symmetric. As observed before it follows that it is constant. Hence $C=J^{\star} $ up to  multiplication by a constant of modulus $1$.
\end{proof}

\section{Conjugations preserving $S$-invariant subspaces of $H^2$}
Beurling's theorem says that  all invariant subspaces for the unilateral shift $S$ are of the form $\theta H^2$ with $\theta$ inner.
We will now investigate conjugations in $L^2$ which preserve subspaces of this form. Since $C_{\theta}$ transforms $\theta H^2$ onto $\overline{z H^2}$, the operator $$C_{\theta}J^{\star}C_{\theta}=M_{\theta}J^{\star}M_{\overline{\theta}}$$
is an example of such a conjugation. Note that
$$(C_{\theta}J^{\star}C_{\theta})M_z=M_z(C_{\theta}J^{\star}C_{\theta}).$$

Let $\alpha$, $\theta$ be two inner functions. Then the operator $C_{\theta}J^{\star}C_{\alpha}\ \colon\ L^2\rightarrow L^2$ is an antilinear isometry which maps $\alpha H^2$ onto $\theta H^2$ and commutes with $M_z$. This operator however does not have to be an involution.

\begin{lemma}\label{L9}
Let $\alpha$, $\theta$ be two inner functions. The operator $C_{\theta}J^{\star}C_{\alpha}$ is an involution (and hence a conjugation in $L^2$) if and only if the function
$\theta\overline{\alpha^{\#}}$ is symmetric
(or equivalently $\alpha\alpha^{\#}=\theta\theta^{\#}$).
\end{lemma}
\begin{proof}
 Note that by Proposition \ref{ppp},
$$(C_{\theta}J^{\star}C_{\alpha})(C_{\theta}J^{\star}C_{\alpha})=C_{\theta}C_{\alpha^{\#}}C_{\theta^{\#}}C_{\alpha}=M_{\theta\overline{\alpha^{\#}}}M_{\theta^{\#}\overline{\alpha}}=M_{\theta\overline{\alpha^{\#}}\theta^{\#}\overline{\alpha}}.$$
Therefore $C_{\theta}J^{\star}C_{\alpha}$ is an involution if and only if
$$\theta\overline{\alpha^{\#}}\theta^{\#}\overline{\alpha}=1\quad\text{a.e. on }\mathbb{T}, \text{ i.e., } \theta\theta^{\#}=\alpha\alpha^{\#},$$
which means that
$$(\theta\overline{\alpha^{\#}})(z)
=(\overline{\theta^{\#}}\alpha)(z)=\theta(\overline{z})\overline{\alpha^{\#}(\bar z)}=(\theta\overline{\alpha^{\#}})(\overline{z})\quad\text{a.e. on }\mathbb{T}.
$$
\end{proof}
The theorem bellow characterizes all $M_z$--commuting conjugations mapping one $S$--invariant subspaces into another $S$--invariant subspace. \begin{theorem}\label{THMC}
		Let $\theta$ and $\alpha$ be two inner functions and let $C$ be a conjugation in $L^2$ such that $CM_z=M_zC$. Then $C(\alpha H^2)\subset \theta H^2$ if and only if $\theta\theta^{\#}\leqslant \alpha\alpha^{\#}$ and $C=C_{\beta}J^{\star}C_{\alpha}$, where  $\beta$ is an inner function such that $\theta\leqslant \beta$, $\beta\beta^{\#}=\alpha{\alpha^{\#}}$. Moreover, in that case $C(\alpha H^2)=\beta H^2$.
	\end{theorem}%

	Let $\alpha$ be a fixed inner function. By Lemma \ref{L9}, for each inner function $\beta$ with $\beta\beta^{\#}=\alpha\alpha^{\#}$ there exists an $M_z$--commuting conjugation $C$ which maps $\alpha H^2$ onto $\beta H^2$, namely $C=C_{\beta}J^{\star}C_{\alpha}$. On the other hand, if $\beta$ is an inner function and there exists an $M_z$--commuting conjugation $C$ which maps $\alpha H^2$ onto $\beta H^2$, then by Theorem \ref{THMC}, $\beta\beta^{\#}\leqslant\alpha{\alpha^{\#}}$ and $C=C_{\gamma}J^{\star}C_{\alpha}$ for some inner function $\gamma$ such that $\beta\leqslant \gamma$, $\gamma\gamma^{\#}=\alpha\alpha^{\#}$. In particular, $C(\alpha H^2)=\gamma H^2=\beta H^2$ and so $\gamma$ is a constant multiple of $\beta$, $\beta\beta^{\#}=\alpha\alpha^{\#}$.
	
	It follows from the above that Lemma 5.3 characterizes all possible spaces of type $\beta H^2$ such that for a given $S$--invariant subspace $\alpha H^2$ there is an $M_z$--commuting conjugation mapping $\alpha H^2$ onto $\beta H^2$.
%

\begin{lemma}
Let $\alpha$ be a nonconstant inner function. Then \begin{multline}\qquad\{\beta: \beta \text{ is inner, }\alpha\alpha^{\#}= \beta\beta^{\#} \}\\= \{\lambda\, uv^{\#}: u, v \text{ are inner, } \alpha=uv, \lambda\in\mathbb{T}\}.\qquad \end{multline}
\end{lemma}

For two inner functions $\alpha$ and $\beta$ denote by $\alpha \wedge\beta$ the greatest common divisor of $\alpha$ and $\beta$. We will write $\alpha \wedge\beta=1$ if the only common divisor of $\alpha$ and $\beta$ is a constant function.

\begin{proof}
Note that for $\alpha=uv$ and $\beta=\lambda uv^{\#}$ we have $\alpha\alpha^{\#}=\beta\beta^{\#}$, hence one inclusion is proved.
For the other inclusion let $u=\alpha\wedge\beta$ and we can write $\alpha=uv$ and $\beta=u v_1$. From the condition $\alpha\alpha^{\#}=\beta\beta^{\#}$ it follows that
$$uvu^{\#}v^{\#}=uv_1u^{\#}v_1^{\#}.$$
Hence
$vv^{\#}=v_1v_1^{\#}$. Since $v\wedge v_1=1$, we have that $v$ divides $v_1^{\#}$ and $v_1$ divides $v^{\#}$ and vice--versa. Thus we can take $v_1=\lambda  v^{\#}$ with $\lambda\in\mathbb{T}$, and so $\beta= \lambda uv^{\#}$.
\end{proof}
\begin{proof}[Proof \!of Theorem \ref{THMC}]
Assume firstly that $CM_z=M_zC$ and $C(\alpha H^2)\subset \theta H^2$. By Theorem \ref{c1}, $C=M_{\psi}J^{\star}$ for some unimodular symmetric function $\psi\in L^{\infty}$. In particular,
$$\psi\alpha^{\#}=M_{\psi}J^{\star}(\alpha)=C(\alpha)\in \theta H^2,$$
and there exists $u\in H^2$ such that $\psi\alpha^{\#}=\theta u$. Note that $u$ must be inner and so $\psi=\beta \overline{\alpha^{\#}}$ with $\beta=\theta u$, $\theta\leqslant\beta$. Clearly $\beta\overline{\alpha^{\#}}$ is symmetric, i.e., $\beta\beta^{\#}=\alpha\alpha^{\#}$.  Hence $\theta\theta^{\#}\leqslant\alpha\alpha^{\#}$.

Assume now that $\theta\theta^{\#}\leqslant\alpha\alpha^{\#}$, and let $\alpha= \alpha_1\cdot (\alpha\wedge\theta)$ and $\theta=\theta_1\cdot (\alpha\wedge\theta)$. Since $(\alpha\wedge\theta)^{\#}=\alpha^{\#}\wedge\theta^{\#}$, we get $\alpha^{\#}= \alpha_1^{\#}\cdot (\alpha^{\#}\wedge\theta^{\#})$ and $\theta^{\#}=\theta_1^{\#}\cdot (\alpha^{\#}\wedge\theta^{\#})$. Note also that $\theta_1\theta_1^{\#}\leqslant \alpha_1\alpha_1^{\#}$ and $\theta_1\wedge\alpha_1=1$, so $\theta_1\leqslant \alpha_1^{\#}$.
 Thus
$$\frac{\alpha\alpha^{\#}}{\theta\theta^{\#}}=\frac{\alpha_1\alpha_1^{\#}}{\theta_1 \theta_1^{\#}}=\frac{\alpha_1^{\#}}{\theta_1}\frac{\alpha_1}{\theta_1^{\#}}=u u^{\#},$$
where $u=\frac{\alpha_1^{\#}}{\theta_1}$ is an inner function. Now we may take $\beta=\theta u$.
Since $\theta\leqslant\beta$ and $\beta\beta^{\#}=\alpha{\alpha^{\#}}$, by Lemma \ref{L9} and by Proposition \ref{ppp}, $C=M_{\beta \overline{\alpha^{\#}}}J^{\star}=C_{\beta}J^{\star}C_{\alpha}$ is a conjugation which maps $\alpha H^2$ onto $\beta H^2\subset \theta H^2$.
%

\end{proof}

\begin{corollary}
Let $\theta$ be an inner function and let $C$ be an $M_z$--commuting conjugation in $L^2$. Then
\begin{enumerate}
  \item $C(\theta H^2)\subset \theta H^2$ if and only if $C=\lambda C_{\theta}J^{\star}C_{\theta}$ with $\lambda\in \mathbb{T}$;
  \item $C(\theta H^2)\subset \theta^{\#} H^2$ if and only if $C=\lambda J^{\star}$ with $\lambda\in \mathbb{T}$.
\end{enumerate}
\end{corollary}
\begin{proof}
By Theorem \ref{THMC},  $C(\theta H^2)\subset \theta H^2$ if and only if there exists an inner function $\beta$ such that
$\theta\leqslant\beta$ and $\beta\beta^{\#}=\theta{\theta^{\#}}$. This is only possible if $\beta$ is constant multiple of $\theta$ and (1) is proved. The proof of (2) is similar.
\end{proof}

%

Note that by Theorem \ref{THMC} (Lemma \ref{L9}, actually) if $\theta\overline{\alpha^{\#}}$ is symmetric, then there exists an $M_z$--commuting conjugation from $\alpha H^2$ into $\theta H^2$. The following example shows that in that case there may be no such conjugation between the corresponding model spaces $K_{\alpha}$ and $K_{\theta}$.
\begin{example}
Fix $a,b\in\mathbb{D}$ such that $a\neq b$, $a\neq \overline{a}$ and $b\neq \overline{b}$, and put
$$\alpha(z)=\tfrac{a-z}{1-\overline{a}z}\ \tfrac{b-z}{1-\overline{b}z}\qquad\text{and}\qquad \theta(z)=\tfrac{a-z}{1-\overline{a}z}\ \tfrac{\overline{b}-z}{1-{b}z}.$$
Then
$$\alpha^{\#}(z)=\tfrac{\overline{a}-z}{1-{a}z}\ \tfrac{\overline{b}-z}{1-{b}z}\qquad\text{and}\qquad \theta^{\#}(z)=\tfrac{\overline{a}-z}{1-{a}z}\ \tfrac{{b}-z}{1-\overline{b}z}$$
and so $\alpha\alpha^{\#}=\theta\theta^{\#}$. 
 Thus there exists an $M_z$--commuting conjugation from $\alpha H^2$ onto $\theta H^2$. In this case however neither $\alpha\leqslant\theta^{\#}$ nor $\theta\leqslant\alpha^{\#}$, so by Theorem \ref{bb1} no $M_z$--commuting conjugation between $K_{\alpha}$ and $K_{\theta}$ exists. Here also neither $\alpha\leqslant\theta$ nor $\theta\leqslant\alpha$, and so by Theorem \ref{t1} no $M_z$--conjugation between $K_{\alpha}$ and $K_{\theta}$ exists.
\end{example}

Finally, consider $M_z$--conjugations preserving $S$--invariant subspaces.

\begin{proposition}\label{THMC1}
Let $\theta$ and $\alpha$ be two inner functions.
There are no $M_z$--conjugations in $L^2$ which map $\alpha H^2$ into $\theta H^2$
\end{proposition}
\begin{proof}
If $C$ was such a conjugation, then by Theorem \ref{c11}, $C=M_{\psi}J$ for some unimodular function $\psi\in L^{\infty}$ and, in particular,
$$C(\alpha)=\psi\overline{\alpha}=\theta g\quad\text{for some }g\in H^2.$$ Clearly $g$ must be an inner function and $\psi =\alpha\theta g$. Then, for every $h\in H^2$,
$$C(\alpha h)=\alpha\theta g\overline{\alpha h}=\theta g\overline{ h}\in \theta H^2,$$
and so $g\overline{h}\in H^2$. It follows that $g=0$ and $C(\alpha)=0$ which is a contradiction.
\end{proof}
%
%
%
%

\section{Conjugations and truncated Toeplitz operators}
For $\varphi\in L^2$ define the {\it truncated Toeplitz operator} $A^\theta_\varphi$ by
$$A^\theta_\varphi f=P_\theta (\varphi f), \text{ for } f\in H^\infty\cap \kdt,$$
were $P_\theta\colon L^2\to \kdt$ is the orthogonal projection (see \cite{Sarason}). The operator $A^\theta_\varphi$ is closed and densely defined, and if it is bounded, it admits a unique bounded extension to $\kdt$. The set of all bounded truncated Toeplitz operators on $\kdt$ is denoted by $\mathcal{T}(\theta)$. Note that $A^\theta_\varphi\in \mathcal{T}(\theta)$ for $\varphi\in L^\infty$. It is known that every operator from $\mathcal{T}(\theta)$ is $C_{\theta}$--symmetric (see \cite[Lemma 2.1]{Sarason}).

Observe that if $k\geqslant 0$, then the conjugation $C_{k,l}$ defined by \eqref{ckl} satisfies neither
$M_zC_{k,l}=C_{k,l}M_z$ nor $SC_{k,l}=C_{k,l}S$. 
However, for $0\leqslant n<k$ and
$\theta(z)=z^n$,
$$C_{k,l}(K_\theta)=K_\theta\quad\text{and}\quad
A_z^{\theta}C_{k,l}=C_{k,l}A_z^{\theta}$$ (since here
$C_{k,l|K_{\theta}}=J^{\star}_{|K_{\theta}}$ and
$\theta^{\#}=\theta$).


Theorem below characterizes conjugations intertwining truncated shifts $ A_z^{\theta}$ and $A_z^{\theta^{\#}}$.
\begin{theorem}\label{c3}
Let $\theta$ be a nonconstant inner function and
let $C$ be a conjugation in $L^2$ such that $C(K_\theta)\subset K_{\theta^{\#}}$.
Then the following are equivalent:
\begin{enumerate}
\item $A_{\varphi^{\#}}^{\theta^{\#}} C=C A_{\varphi}^{\theta}$ on $K_{\theta}$ for all $\varphi\in H^\infty$,
\item $ A_z^{\theta^{\#}} C=C A_z^{\theta}$ on $K_{\theta}$,
 \item there is a function $\psi\in H^{\infty}$ such that
 $C_{|K_{\theta}}= J^{\star}A_{\psi}^{\theta}$ and $A_{\psi}^{\theta}$ is an isometry,
 \item there is a function $\psi'\in H^{\infty}$ such that
 $C_{|K_{\theta}}= A_{\psi'}^{\theta^{\#}}J^{\star}_{|K_{\theta}}$ and $A_{\psi'}^{\theta^{\#}}$ is an isometry.
 \end{enumerate}
\end{theorem}

\begin{proof}
We will only prove that $(2)\Rightarrow (3)$.
Since $J^{\star}(K_{\theta^{\#}})=K_\theta$
and $J^{\star}A_{\varphi^{\#}}^{\theta^{\#}} = A_{\varphi}^{\theta}J^{\star}$ for all $\varphi\in H^{\infty}$ (see \cite[Lemma 4.5]{CGW}), we have
$$J^{\star}CA_z^{\theta}=J^{\star}A_z^{\theta^{\#}}C=A_z^{\theta}J^{\star}C$$
on $K_{\theta}$ and so $J^{\star}C_{|K_{\theta}}=A_{\psi}^{\theta}$ for some $\psi\in H^{\infty}$ (\cite[Theorem 14.38]{FM}). Hence $A_{\psi}^{\theta}$ is an isometry and
$$C_{|K_{\theta}}=J^{\star}A_{\psi}^{\theta}.$$
\end{proof}

It is much more restrictive if $\theta^{\#}=\theta$.

\begin{proposition}
Let $\theta$ be an inner function such that $\theta^{\#}=\theta$ and
let $C$ be a conjugation in $K_\theta$. Then the following are equivalent:
\begin{enumerate}
\item $A_{\varphi^{\#}}^{\theta}C=CA_{\varphi}^{\theta}$ for all $\varphi\in H^\infty$,
\item $ A_z^{\theta} C= CA_z^{\theta}$,
 \item 
 $C=\lambda J^{\star}_{|K_{\theta}}$ with $\lambda\in\mathbb{T}$.
 \end{enumerate}
\end{proposition}
\begin{proof}
Implications $(1)\Rightarrow (2)$ and $(3)\Rightarrow (1)$ are clear.
To prove $(2)\Rightarrow (3)$ apply Theorem \ref{c3} to the conjugation $\tilde{C}$ in $L^2$ defined by $$\tilde{C}=C\oplus C_{\theta}\ \colon\ K_{\theta}\oplus (K_{\theta})^{\perp}\rightarrow K_{\theta}\oplus (K_{\theta})^{\perp}.$$ It follows that $\tilde{C}_{|K_{\theta}}=C=J^{\star}A_{\psi}^{\theta}$ for $\psi\in H^{\infty}$ such that $A_{\psi}^{\theta}$ is an isometry. Since $C(K_{\theta})=K_{\theta}$ and $J^{\star}(K_{\theta})=K_{\theta^{\#}}=K_{\theta}$, we see that $A_{\psi}^{\theta}=J^{\star}C$ maps $K_{\theta}$ onto $K_{\theta}$ and is in fact unitary. Thus we have
$$A_{\overline{\psi}}^{\theta}A_{\psi}^{\theta}=A_{\psi}^{\theta}A_{\overline{\psi}}^{\theta}=I_{K_{\theta}}.$$
On the other hand, $C^2=I_{K_{\theta}}$ so
$$C^2=J^{\star}A_{\psi}^{\theta}J^{\star}A_{\psi}^{\theta}=A_{\psi^{\#}}^{\theta}A_{\psi}^{\theta}=A_{\psi}^{\theta}A_{\psi^{\#}}^{\theta}=I_{K_{\theta}}.$$
Hence $A_{\overline{\psi}}^{\theta}=A_{\psi^{\#}}^{\theta}$ and
$A_{\overline{\psi}-\psi^{\#}}^{\theta}=0$, which gives
$\overline{\psi}-\psi^{\#}\in \overline{\theta H^2}+\theta H^2$ (see \cite{Sarason}). In
other words, $\overline{\psi}-\psi^{\#}=\overline{\theta h_1}+\theta
h_2$ for some functions $h_1, h_2\in H^2$. Thus there exists a
constant $\lambda$ such that $${\psi}-\theta
h_1=\overline{\psi^{\#}+\theta h_2}=\overline{\lambda}.$$ We now have
$$A_{\psi}^{\theta}=A_{\theta
h_1+\overline{\lambda}}^{\theta}=\overline{\lambda} I_{K_{\theta}}.$$
Moreover
$\lambda\in\mathbb{T}$, since $A_{\psi}^{\theta}$ is unitary. Hence
\[C=J^{\star}A_{\psi}^{\theta}={\lambda} J^{\star}_{|K_{\theta}}.\]
\end{proof}

%
%


 Now we characterize conjugations intertwining the truncated shifts $ A_z^{\theta}$ and $A_{\bar z}^{\theta}$.
\begin{theorem}\label{c4}
Let $\theta$ be an inner function and let $C$ be a conjugation in
$K_\theta$. Then the following are
equivalent:
\begin{enumerate}
\item $A_{\varphi}^{\theta} C=C A_{\overline{\varphi}}^{{\theta}}$ for all $\varphi\in H^\infty$,
\item $ A_z^{\theta} C=C A_{\bar z}^{\theta}$ ,
 \item there is a function $\psi\in H^\infty$ such that
 $C=A_\psi^{\theta}C_{\theta}$ and $A_\psi^{\theta}$
 is unitary,
 \item there is a function $\psi^\prime\in H^\infty$ such that
 $C=C_{\theta}A_{\overline{\psi^\prime}}^{\theta}$ and $A_{\psi^\prime}^{\theta}$
 is unitary.
 \end{enumerate}
\end{theorem}
\begin{proof}
Let us start with $(2)\Rightarrow (3)$.
Since $A_z^{\theta}$ is $C_{\theta}$--symmetric,
$$A_z^{\theta} CC_{\theta}=CA_{\bar
z}^{\theta}C_{\theta}=CC_{\theta}A_z^{\theta}.$$  Hence, by \cite[Proposition 1.21]{berc},
$CC_{\theta}=A_\psi^{\theta}$ for some $\psi\in H^\infty$. Clearly,
$A_\psi^{\theta}$ is unitary and
$$C=A_\psi^{\theta}C_{\theta}=C_{\theta}(A_\psi^{\theta})^{*}=C_{\theta}A_{\overline{\psi}}^{\theta}.$$

To prove that $(4)\Rightarrow (1)$ note that,
since $A_{\overline{\psi^\prime}}^{\theta}$ and
$A_{\overline{\varphi}}^{\theta}$ commute, we have
$$A_{\varphi}^{\theta}C=A_{\varphi}^{\theta}C_{\theta}A_{\overline{\psi^\prime}}^{\theta}
=C_{\theta}A_{\overline{\varphi}}^{\theta}A_{\overline{\psi^\prime}}^{\theta}=C_{\theta}A_{\overline{\psi^\prime}}^{\theta}A_{\overline{\varphi}}^{\theta}=CA_{\overline{\varphi}}^{\theta}=C(A_{{\varphi}}^{\theta})^{*}.$$
All other implications are straightforward.
\end{proof}

\begin{corollary}
If $C$ is a conjugation in $K_\theta$
and every $A\in \mathcal{T}(\theta)$ is $C$--symmetric, then
$C=A_\psi^{\theta}C_{\theta}$ for some $\psi\in
H^\infty$ such that $A_\psi^{\theta}$ is unitary.
\end{corollary}

For a complete description of unitary operators from $\mathcal{T}(\theta)$ see \cite[Proposition 6.5]{SED}.


\begin{thebibliography}{99}
\bibitem{berc}
{Bercovici H}. {Operator Theory and Arithmetic in $H^{\infty}$}.
Mathematical Surveys and Monographs, vol. 26, Providence: Amer Math Soc, RI, 1988


\bibitem{CKP}
{C\^{a}mara C, Kli\'s-Garlicka K and Ptak M}.
{Asymmetric truncated Toeplitz operators and conjugations}.
Filomat, to appear.

\bibitem{CFT}
Chevrot N, Frician E, Timotin D.
{The characteristic function of a complex symmetric contraction}.
Proc Amer Math Soc, 2007, 135: 2877--2886


\bibitem{CGW}
{Cima J A, Garcia S R,  Ross W T, et al}.
{Truncated Toeplitz operators:  spatial isomorphism, unitary equivalence, and similarity}.
 Indiana Univ Math J, 2010, 59: 595--620



\bibitem{FM}
{Fricain E,  Mashreghi J}.
{The theory of $\mathcal{H}(b)$ spaces, Vol. 1}.
Cambridge: Combridge University Press,  2016

\bibitem{GMR}
 Garcia S R, Mashreghi J, Ross W T.
{Introduction to model spaces and their operators.}
Cambridge: Cambridge University Press, 2016

\bibitem{GP}
{Garcia S R,  Putinar M}. {Complex symmetric operators and applications}. Trans Amer Math Soc, 2006, 358: 1285--1315
\bibitem{GPP}
{Garcia S R, Prodan E, Putinar M}. {Mathematical and physical aspects of complex symmetric operators}. J Phys A: Math Theor, 2014, 47: 353001

\bibitem{GP2}
{Garcia S R, Putinar M}. {Complex symmetric operators and applications II}. Trans Amer Math Soc, 2007, 359: 3913--3931




\bibitem{KoLee18}
 {Ko E,   Lee J E}.
  {Remark on complex symmetric operator matrices}.
Linear and Multilinear Algebra, 2018, DOI: 10.1080/03081087.2018.1450350



\bibitem{MAR}
 { Mart{\'{i}}nez-Ave\~{n}dano R A, Rosenthal P}.
  {An Introduction to Operators on the Hardy-Hilbert Space}.
New York: Springer Science+Business Media, LLC, 2007

\bibitem{NF}
{Sz.-Nagy B, Foias C F,  Bercovici H, et al}.
{Harmonic Analysis of Operators on a Hilbert Space}, second edition, London: Springer, 2010

\bibitem{RR}
{ Radjavi H,  Rosenthal P}. {Invariant Subspaces}.
 New York: Springer, 1973

\bibitem{Sarason}
{Sarason D}.
{Algebraic properties of truncated Toeplitz operators}. Oper Matrices, 2007, 1: 491--526

\bibitem{SED}
{ Sedlock N A}.
{Algebras of truncated Toeplitz operators}. Oper Matrices, 2011, 5: 309--326
\end{thebibliography}
\end{document}